\documentclass{amsart}
\usepackage{amssymb, amsmath, amstext, amsopn, amsthm, amscd, amsxtra, amsfonts}
\usepackage{bm}
\usepackage[theoremdefs, final]{allergy}

\usepackage{enumerate}
\usepackage{pgfplots}

\usetikzlibrary{external}
\usepackage{xcolor}
\usepackage{comment}
\usepackage[backend=biber, style=alphabetic]{biblatex}

\bibliography{biblio.bib}
\newcommand{\ud}{\mathrm{d}}
\newcommand{\mat}[1]{\mathbf{#1}}

\newcommand{\R}{\mathbf{R}}
\renewcommand{\SS}{\mathbf{S}}
\newcommand{\RR}{\mathbb{R}}
\newcommand{\CC}{\mathbb{C}}
\newcommand{\p}{\mathbf{p}}
\newcommand{\q}{\mathbf{q}}

\newcommand{\w}{\mathbf{w}}
\newcommand{\y}{\mathbf{y}}
\newcommand{\z}{\mathbf{z}}
\newcommand{\PP}{\mathbf{P}}
\newcommand{\ZZ}{\mathbf{Z}}
\newcommand{\WW}{\mathbf{W}}
\newcommand{\veclambda}{\boldsymbol{\lambda}}
\newcommand{\defeq}{\coloneqq}
\newcommand{\from}{\colon}

\newcommand{\bigO}{\mathcal{O}}
\newcommand{\tran}{\mathsf{T}}

\begin{document}

\title{Collective Symplectic integrators on \texorpdfstring{$S_2^n \times T^{\ast} \RR^m$}{S2n x T*Rm}}

\author{Geir Bogfjellmo}
\address{ICMAT, 28049 Madrid SPAIN}
\email{geir.bogfjellmo@icmat.es}
\subjclass[2010]{Primary 65P10 ; Secondary 65L06, 37M15}

\begin{abstract}A novel symplectic integrator for Hamiltonian equations on $S_2^n \times
  T^{\ast} \RR^m$ is developed and studied.
  Partitioned Runge--Kutta methods for Hamiltonian systems on products of
  Hamiltionian manifolds are studied, specifically, algebraic conditions for
  their symplecticity are derived.
\end{abstract}

\maketitle

\section{Introduction}
When a differential equation inhibits geometrical properties, it is considered
advantageous that numerical approximations to the equation inhibits the same
properties.
One such geometrical property is the \emph{symplecticity} inhibited by Hamiltonian systems.
In general, a \emph{symplectic} space is a manifold $M$, equipped with a closed
two-form $\omega$.
A differential equation
\[\frac{\ud}{\ud t}z=X(z),\]
where $X$ is a vector field over $M$, is symplectic if
the Lie derivative of $\omega$,
\[\mathcal{L}_X \omega = \ud \imath_X \omega= 0\]

Numerical approximations preserving symplecticity are known as symplectic integrators.
Symplectic integrators for ordinary differential equations evolving on vector
spaces have been studied by many authors, see for instance \cite{HLW, CSS}
and the references therein.

The situation for non-flat geometries is more complicated, and usually relies on
the particular geometry.

This paper studies a special case of non-flat symplectic space, the product of
copies of $S^2$ and the canonical symplectic space $T^\ast \RR^M$. This space is
of special interest in spin-lattice dynamics.
\section{Hamiltonian systems on \texorpdfstring{$S_2^n \times T^{\ast} \RR^m$}{S2n x T*Rm}}
Hamiltonian systems evolving on $M= S_2^n \times T^{\ast} \RR^m$ arise in e.g.
spin-lattice-electron (SLE) equations. See, e.g. \cite{MaDuWo08, MaDuWo12, ASDBook}.
In these systems, each particle $i$ state is given by a position $\q_i\in \RR^3$, a momentum $\p_i \in \RR^3$, and a spin $\w_i\in S_2$.
With $k$ particles, the total state space is thus $S_2^k \times T^\ast
\RR^{3k}$.

We write the state of the system $(\w, \q, \p)$, where
\[\w = \begin{bmatrix} \w_1 \\ \w_2 \\ \vdots \\ \w_k \end{bmatrix} \in S_2^k,
  \qquad
  \p = \begin{bmatrix} \p_1 \\ \p_2 \\ \vdots \\ \p_k \end{bmatrix} \in \RR^{3k},
  \qquad
  \q = \begin{bmatrix} \q_1 \\ \q_2 \\ \vdots \\ \q_k \end{bmatrix} \in \RR^{3k}.
\]

The Hamiltonian for SLE-systems is
\begin{equation}
\begin{aligned}
H(\w,\q,\p) &= T_L(\p) + U_L(\q) + H_m(\w,\q)\\
 &= \frac{1}{2}\sum_{i=1}^n \frac{\norm{\p_i}^2}{m_i} + U_L(\q) - \frac{1}{2} \sum_{i,j=1}^n  J_{ij}(\q) \langle \w_i, \w_j\rangle
 \end{aligned}
\label{eq: SLEHam}
\end{equation}
where $m_i$ is the mass of each individual particle, $U_L(\q)$ is a potential
depending on the positions $\q$, and $J_{ij}$ determines the strength of the spin couplings, depending on the positions $\q$.
Typically, $J_{ij}(\q)= J(\norm{q_i-q_j})$, but other functions are possible.

The resulting Hamiltonian equations are
\begin{equation}
\begin{aligned}
\frac{\ud \q_i}{\ud t} &= \frac{\p_i}{m_i} \\
\frac{\ud \p_i}{\ud t} &= - \frac{\partial U_L}{\partial \q_i}(\q) + \frac{1}{2}  \sum_{j,k=1}^n  \frac{\partial J_{jk}(\q)}{\partial \q_i} \langle \w_j, \w_k\rangle\\
\frac{\ud \w_i}{\ud t} &=\w_i \times \left[\sum_j J_{ij}(\q) \w_j\right]
\end{aligned}
\label{eq: SLE}
\end{equation}

For the following, we define the matrix
\[
  \mat{M}= \begin{pmatrix}m_1 \mat{I}_3& \\
    & m_2\mat{I}_3 & \\
    & & \ddots & \\
    & & & m_k \mat{I}_3
    \end{pmatrix}
\]

Symplectic integration of \eqref{eq: SLE} has previously been obtained by splitting methods, see e.g. \cite{OmMrFo01}. These methods rely on a symmetric splitting where each spin is integrated individually.
A disadvantage of this approach is that the spins has to be updated in sequence, limiting the possibilities of parallelization.
Furthermore, the splitting methods are incapable of handling more general
Hamiltonians. For instance, Perera et al. \cite{Per16} introduce an anisotropy
term.

An alternative to the method in the present paper is the method used by Hellsvik
et. al. in \cite{HeThMo18}

In this article, we suggest a novel approach for symplectic integration on $S_2^n \times T^{\ast} \RR^m$.
This approach is based on a partitioned integrator, where the positions and
momenta are integrated with a standard symplectic partitioned Runge--Kutta
method, and the spins are integrated with a collective symplectic integrator on
$S_2^n$. The development of symplectic integrators on $S_2^n$ is due to by
McLachlan, Modin and Verdier \cite{MMV14,MMV15,MMV17}.

The novel integrator is implicit, as opposed to splitting-based methods.

The integrators derived in this approach can in principle handle any Hamiltonian
on $S_2^n \times T^{\ast} \RR^m$.
We will, however, focus on the case where the dependence on the momentum $\p$ can be split of as a quadratic kinetic term $T_L$.

The geometry needed for these integrators is as follows:
\begin{enumerate}
\item The symplectic manifold $M$ is embedded into a Poisson manifold $P$, such that the image of $M$ is a symplectic leaf.
\item $P$ has a full realization as a canonical symplectic manifold $N\simeq T^{\ast} \RR^d$, i.e. there exists an onto Poisson map $\psi \from N \to P$.
\end{enumerate}

For the integrators to be well-defined, it is necessary to extend the Hamiltonian $H$ into a function $\bar{H}\from P\to \RR$.
This extension is not-unique, however we will fix it to a ``canonical'' choice,
following \cite{MMV17}.
The dynamics on the top symplectic manifold $N$ are defined by the pulled-back
Hamiltonian $\psi\circ \bar{H}$.

\section{Symplectic and Poisson structures}
To proceed, we need to define the various symplectic and Poisson structures involved.

\subsection{Symplectic structure on \texorpdfstring{$M$}{M}}
Let $M= \SS \times T^\ast V$, where $\SS= S_2^n$ and $V= \RR^m$. (In SLD $m=dn$,
where $d$ is the dimension of the lattice.)

Both $(\SS, \omega^S)$ and $(T^\ast V, \omega^V)$ are symplectic manifolds with
$\omega^V$ the canonical two-form and
\[ \omega^S = \sum_i \ud A_i,
\]
where $A_i$ is the standard area form on the $i$th sphere.

Let $\omega$ be the product symplectic form
\[\omega = \pi_1^\ast \omega^V + \pi_2^\ast \omega^{S}
\]
where $\pi_1\from M \to T^\ast V$ and $\pi_2 \from M \to \SS$ are the canonical projections.
$(M, \omega)$ is the direct product symplectic manifold of $\SS$ and $T^\ast V$.

\subsection{Poisson structure on \texorpdfstring{$P$}{P}}
McLachlan, Modin and Verdier \cite{MMV17} obtained symplectic integrators on $\SS$ by embedding $\SS= S_2^n$ as a symplectic leaf in the Poisson manifold $\RR^{3n}$.
A straightforward generalization of this is to embed $M=T^\ast V \times \SS$ as a symplectic leaf in a Poisson manifold.

Let $P = T^\ast V\times \R$, where $\R= \RR^{3n}$.

$(T^\ast V, \{\cdot,\cdot\}_V)$ is a Poisson manifold whose Poisson bracket is induced by the symplectic structure,
\[
\{f,g\}_V = (\omega^V)^{-1}(\ud f \wedge \ud g)
\]
where $(\omega^V)^{-1}$ is the two-vector obtained by inverting the symplectic form.

The Poisson bracket on $\R=(\RR^{3})^n$, which we denote $\{\cdot,\cdot\}_S$, is the sum of Poisson brackets on each copy of $\RR^3$,
\[
\{f,g\}_S(\w) = \kappa(\w)(\ud f \wedge \ud g) =\sum_i \left\langle \w_i, \left[\frac{\partial f(\w)}{\partial \w_i}, \frac{\partial g(\w)}{\partial \w_i}\right]\right\rangle.
\]

On $P$, we obtain a Poisson bracket by taking the sum of the brackets on each component,
\begin{equation}
  \begin{aligned}
    \{f,g\}(\y,\w) = &(\omega^V)^{-1} (\ud_\y f(\y,\w) \wedge \ud_\y g(\y,\w)) \\
    &+ \kappa(\w)(\ud_\w f(\y,\w) \wedge \ud_\w g(\y,\w)),
    \end{aligned}
\label{eq: PoissP}
\end{equation}
for all $\y\in T^\ast V, \w\in \R$.

In the above equation, $\ud_\y$ and $\ud_\w$ denote the partial differentials, e.g.
\[\ud_\y f = \sum_i \frac{\partial f}{\partial \y_i} \ud \y_i.\]

Using canonical coordinates $\y=(\p,\q)$, the full form of the Poisson bracket is
\[
\{f,g\}(\p,\q,\w) = \sum_{i=1}^m \left(\frac{\partial f}{\partial q_i} \frac{\partial g}{\partial p_i} -  \frac{\partial f}{\partial p_i} \frac{\partial g}{\partial q_i}\right) + \sum_{j=1}^n \left\langle \w_j, \left[\frac{\partial f}{\partial \w_j}, \frac{\partial g}{\partial \w_j}\right]\right\rangle.
\]

\begin{proposition}
$M$ is a symplectic leaf in $P$.
\end{proposition}
\begin{proof}
$\SS$ is a symplectic leaf in $\R$ and $T^\ast V$ is a symplectic manifold.
\end{proof}

Having embedded $M$ into $P$, we also need to extend vector fields on $M$ to vector fields on $P$.
Taking a leaf from \cite{MMV17}, we do this by
letting the Hamiltonian and vector fields be constant on rays, i.e. sets of the
form
$
\{(\y,\veclambda \odot \w) \colon  \veclambda =(\lambda_1, \dotsc, \lambda_n)\in \RR^n_+ \} \subset P,
$
where $\y \in T^\ast V, \w \in \R$ and
$
\veclambda\odot \w = (\lambda_1\w_1, \dotsc, \lambda_n\w_n).
$

We define a projection map $\rho_1\from \R \to \SS$ by
\[
\rho_1(\w_1, \w_2, \dotsc, \w_n) \defeq \left( \frac{\w_1}{\norm{\w_1}}, \dotsc, \frac{\w_n}{\norm{\w_n}}\right),
\]
and a projection map $\rho \from P \to M$ by
\[
\rho(\y, \w)= (\y, \rho_1(\w)) 
\]

It is a simple exercise to show that if $H\from M \to \RR$ is a Hamiltonian with
associated vector field, $X_H$, then $\bar{H}=H\circ \rho \from P \to \R$
is a Hamiltonian on $P$ with associated Poisson vector field
\[X_{\bar{H}}(\y,\w) = X_H(\rho(\y, \w)).
\]
We call this vector field the \emph{extension of} $X_H$ \emph{to} $P$.

Notice that $X_{\bar{H}}$ is tangent to every symplectic leaf, not only to $M$.

In particular, for the Hamiltonian of interest \eqref{eq: SLEHam}, the extended Hamiltonian takes the form
\begin{equation}
\bar{H}(\w, \p,\q) = T_L(\p) + H_1(\w,\q)
\label{eq: barH}
\end{equation}
where $H_1(\w,\q) = U_L(\q) + H_m(\rho_1(\w),\q)$.


\subsection{Realization of \texorpdfstring{$P$}{P}}
\begin{definition}
A \emph{realization} of a Poisson manifold $(P, \{\cdot,\cdot\})$ is a symplectic manifold $(N, \omega_N)$ together with a Poisson map $\psi \from N\to P$.
The realization is called \emph{full} if it is surjective and \emph{canonical} if $N \simeq T^\ast \RR^d$ for some $d$.
\end{definition}

A realization of $P=(\RR^3)^n \times T^\ast V$ is obtained by using the Hopf fibration map for each copy of $\RR^3$.

\begin{proposition}
Let
$N=\CC^{2n} \times T^\ast \RR^m$, equipped with the canonical symplectic structure.
We write a point in $N$ as $(\z_1, \z_2, \p,\q)$ where
\[
\z_{i} = (z_{i}^1, z_{i}^2, \dotsc, z_{i}^n), \qquad \text{for }i=1,2
\]
and $(\p,\q)\in T^\ast \RR^m$.

Let a map $\psi\from N\to P$ be defined by
\[
\begin{aligned}
\psi(\z_1, \z_2, \p_{\q}) &= (\mathbf{J}(\z_1, \z_2), \p, \q)\\
                          &= \left(J(z_1^1, z_2^1), \dotsc, J(z_1^n, z_2^n),
                            \p, \q\right),
\end{aligned}
\]
where
$J\from \CC \times\CC\to \RR^3$ is the Hopf fibration map
\begin{equation}
J(z_1,z_2) = \frac{1}{4}\begin{pmatrix}
                  2\Re(z_1z_2^\ast)\\
                  2\Im(z_1z_2^\ast)\\
                  \left(|z_1|^2-|z_2|^2\right)
                  \end{pmatrix}.
\label{eq: Hopf}
\end{equation}
\label{prop: realization}

$(N, \psi)$ is a full, canonical realization of $P$.
\end{proposition}

\begin{proof} Direct products of Poisson maps are Poisson. $\psi$ is the product of $n$ copies of the Hopf fibration and the identity map on $T^\ast V$. Finally, the Hopf fibration is Poisson \cite{MR09}.
\end{proof}

The pull-back of the Hamiltonian \eqref{eq: barH} has the form
\begin{equation}
\begin{aligned}
\bar{H}\circ \psi(\z,\p,\q) &= T_L(\p) + H_2(\z, \q)\\
&= T_L(\p) + H_1(\mathbf{J}(\z), \q)
\end{aligned}
\label{eq: pullbackH}
\end{equation}

\section{Collective Symplectic integrators}
McLachlan, Modin and Verdier introduced \emph{Collective symplectic integrators} for integration of Poisson systems.
The main idea is to utilize a realization $(N, \psi)$ of the Poisson manifold $P$ and integrate the vector field of the pulled-back Hamiltonian $H\circ \psi$ on $N$, with a symplectic method.
We will denote this vector field by $X_{H\circ \psi}$.

To obtain a symplectic integrator on $P$ (i.e. Poisson and preserves leaves), it
is necessary that the update maps of the integrator maps fibers (of $\psi$) to fibers, and that the preimages of leaves are preserved.

In our case $N=\CC^{2n}\times T^\ast \RR^m \simeq T^\ast \RR^{2n} \times T^\ast \RR^m$.
We use coordinates $(\z,\p,\q)$ on $N$ and assume we use a PRK method
partitioned into these coordinates (i.e. each of the components $\z,\p,\q$ are
integrated with (possibly different) RK methods.)
\begin{lemma}
A PRK method, when applied to a lifted vector field $X_{H\circ \psi},$ maps fibers to fibers.
\label{lem: fibers}
\end{lemma}
\begin{proof}
The group $U(1)^{\times n}$ acts on $N$ with the following action.
Write an  element in $U(1)^{\times n}$ as
\[e^{i\bm{\theta}} = (e^{i\theta_1}, e^{i\theta_2}, \dotsc, e^{i\theta_n})
\]
where $\theta_k\in [0, 2\pi]$. The action is given by
\[
e^{i\bm{\theta}}\cdot(\z_1,\z_2,\p,\q) =
(e^{i\bm{\theta}}\odot \z_1, e^{i\bm{\theta}}\odot \z_2, \p, \q).
\]
This action is linear and symplectic. Furthermore, it preserves fibers and is transitive on each fiber.
As the action is symplectic and preserves fibers, it is a symmetry of the lifted vector field.
Since the action is also linear, and only affects one of the components, the
$\z$-component, this symmetry is preserved by the partitioned Runge--Kutta
method, and is also a symmetry of the update map.
Since the action is transitive on each fiber, we can conclude that the update map maps fibers to fibers.
\end{proof}

\begin{lemma}
A PRK method, where the $\z$-component is integrated with a symplectic RK-method, preserves the preimages of leaves.
\label{lem: leaves}
\end{lemma}
\begin{proof}
The symplectic leaves in $P$ are given by $\norm{\w_j}=r_j$, for each $j$.
By properties of the Hopf map \eqref{eq: Hopf}, the preimages in $N$ are given by
\[
|z_1^j|^2 + |z_2^j|^2 = 2r_j
\]
for each $j$, and are invariant under the flow of the lifted vector field.
As quadratic invariants, depending only on $\z$, these are preserved by the PRK
methods if the $\z$-method is symplectic.
\end{proof}

We are almost ready to state the main theorem, except that we need a result on
the symplecticity of Partitioned Runge Kutta methods with three components,
where one component, $\z$, corresponds to a symplectic space, and the two remaining
components, $\p, \q$, correspond to Darboux coordinates of another symplectic space.

The symplecticity of such partitioned methods is interesting in its own right, and the proof of this is presented in Section \ref{sec: NSP}.

\begin{theorem}
Assume the system on $N$ is integrated with a partitioned Runge--Kutta method, where
\begin{itemize}
\item $\z$ is integrated with a symplectic Runge--Kutta method.
\item $(\p, \q)$ is integrated with a symplectic partitioned Runge--Kutta method and
\item The $b$-coefficients of the two methods above coincide.
\end{itemize}
Then the resulting integrator is symplectic. Furthermore, it descends to a symplectic method on $P$, and the descended method restricts to a symplectic method on $M$.
\end{theorem}

\begin{proof}
Sufficient conditions are that the ``upstairs'' integrator on $M$ \cite{MMV17}
\begin{enumerate}[(i)]
\item is symplectic.
\item maps fibers to fibers.
\item preserves preimages of leaves in $P$.
\end{enumerate}

(i) follows from Theorem \ref{thm: PRK} and the remarks following.
(ii) is  Lemma \ref{lem: fibers}.
(iii) is Lemma \ref{lem: leaves}.
\end{proof}

The method used for the numerical tests is a partitioned method where the
$\z$-variable is integrated with the implicit midpoint method and the
$(\p,\q)$-variable are integrated with the St\"{o}rmer--Verlet Scheme. As the
midpoint method is a one-stage method, and the St\"{o}rmer--Verlet method has
two stages, it is necessary to use the reducible two stage method with Butcher tableau
\[\arraycolsep=1.4pt\def\arraystretch{1.5}
\begin{array}{c|cc}
\frac 12 & \frac 14 & \frac 14 \\
\frac 12 & \frac 14 & \frac 14 \\
\hline
        & \frac 12 & \frac 12
\end{array}\]
for the $\z$-coordinate.

In the following equations describing the integrators, we will use $\p,\q,$etc.
for the values before a step of the integrator, and $\tilde{\p},
\tilde{\q},$etc. for the values after a step of the integrator.

The partitioned integrator, applied to the Hamiltonian \eqref{eq: pullbackH} on $N$  can, after identifications, be written:
\begin{equation}
\begin{aligned}
\PP &= \p - \frac{h}{2} \frac{\partial H_2}{\partial \q}(\q, \ZZ)\\
\ZZ &= \z + \frac{h}{4} \mat{J}_{\z}^{-1} \frac{\partial H_2}{\partial \z}(\q, \ZZ)+ \frac{h}{4} \mat{J}_{\z}^{-1} \frac{\partial H_2}{\partial \z}(\tilde{\q}, \ZZ)\\
\tilde{\p} &= \PP - \frac{h}{2} \frac{\partial H_2}{\partial \q}(\tilde{\q}, \ZZ)\\
\tilde{\q} &= \q+h \mat{M}^{-1} \PP\\
\tilde{\z} &= \z + \frac{h}{2} \mat{J}_{\z}^{-1} \frac{\partial H_2}{\partial \z}(\q, \ZZ)+ \frac{h}{2} \mat{J}_{\z}^{-1} \frac{\partial H_2}{\partial \z}(\tilde{\q}, \ZZ).
\end{aligned}
\label{eq: scheme3}
\end{equation}

The $\z$-variable is integrated with the midpoint method.
As shown in \cite{MMV16}, this integrator coincides with the spherical midpoint method for ray-constant vector fields (As we have here, cf.\eqref{eq: barH})

We can thus write the scheme as
\begin{equation}
\begin{aligned}
\PP &= \p - \frac{h}{2} \frac{\partial H_1}{\partial \q}(\q, \WW)\\
\WW &= \rho_1(\w+ \tilde{\w})\\
\tilde{\p} &= \PP - \frac{h}{2} \frac{\partial H_1}{\partial \q}(\tilde{\q}, \WW)\\
\tilde{\q} &= \q+h \mat{M}^{-1} \PP\\
\tilde{\w} &= \w + \frac{h}{2} \left[\WW, \frac{\partial H_1}{\partial \w}(\q, \WW)+ \frac{\partial H_1}{\partial \w}(\tilde{\q}, \WW)\right]
\end{aligned}
\label{eq: scheme4}
\end{equation}

\section{Nonstandard Symplectic partitioned Runge--Kutta methods}
\label{sec: NSP}
For the collective integrators proposed, we integrate a Hamiltonian system where the space is partitioned into a product of two symplectic spaces. We therefore need to establish when a partitioned Runge--Kutta method is symplectic for this partitioning.\footnote{Standard symplectic PRK-methods partition into position and momentum variables}.

The conditions are not specific to our application and are here presented in a more general setting.

We first consider the case when a symplectic manifold is partitioned into $N$ symplectic spaces, and each component is integrated with a (nonpartitioned) Runge-Kutta method.

Consider an ordinary differential equation of the form
\begin{equation}
\begin{aligned}
\frac{\ud y^k}{\ud t} &= f^k(\y), \qquad k=1, \dotsc, N\\
y^k &\in \RR^{n_k}
\end{aligned}
\label{eq: ode}
\end{equation}
where
\[
\y= \begin{bmatrix} y_1 \\ \vdots \\ y_N\end{bmatrix}
\]

Equation \eqref{eq: ode} can be numerically integrated by a partitioned Runge--Kutta method with coefficients $b^k_i, a^{k}_{ij}$, where $k=1, \dotsc, N$ and $i,j=1,\dotsc s$.
When writing down the scheme, we write
\[
\tilde{\y}= \begin{bmatrix} \tilde{y}_1 \\ \vdots \\ \tilde{y}_N\end{bmatrix}
\]
for the updated variables, and
\[
\mathbf{Y}_i= \begin{bmatrix}  Y_1^k\\ \vdots \\ Y_N^k\end{bmatrix}, \qquad\text{$i=1, \dotsc,s$}
\]
for the stage values.
\begin{equation}
\begin{aligned}
\tilde{y}^k&= y^k + h\sum_{i=1}^s b^k_i F^{k}_i & k&=1,\dotsc,N\\
Y_i^k &= y^k + h\sum_{j=1}^s a^k_{ij} F^{k}_j & k&=1,\dotsc,N, \quad i=1,\dotsc,s\\
F^{k}_i &= f^k(\mathbf{Y}_i) & i=1,\dotsc,s
\end{aligned}
\label{eq: scheme}
\end{equation}

We are interested in which of these schemes preserve symplectic forms of the type
\[
\omega= \sum_{k=1}^N \omega_k = \sum_{k=1}^N \ud y^k\wedge  \mat{J_k} \ud y^k.
\]
where each $\mat{J_k}$ is a $n_k \times n_k$ skew-symmetric, nonsingular matrix

A standard application of the variational equation (see e.g. \cite[Chapter VI.4]{HLW}) shows that for preserving symplectic forms of the above type it is sufficient that the integrator preserves all first integrals of the form
\[I(\y) = \sum_{k=1}^N I_k(y_k) = \sum_{k=1}^N B^k(y^k, y^k),\]
where each $B^k$ is a symmetric bilinear function.

\begin{theorem}
If the coefficients satisfy that
\begin{enumerate}[(i)]
\item $b^k_ib^k_j = b^k_i a^k_{ij} + b^k_j a_{ji}$ \qquad for all $i,j,k$ and
\item $b^1_i=b^2_i= \dotsb = b^k_i$ \qquad for all $i$.
\end{enumerate}
then the scheme preserves all invariants of the form $I(\y)=\sum_{k=1}^N
B^k(\y^k, \y^k)$.
\label{thm: RK1}
\end{theorem}
Another way of stating the assumption in the theorem is that each of the Runge--Kutta methods is symplectic in their own right, and their $b$-values all have to agree.

\begin{proof}
Since $I$ is a first integral, it holds that $\sum_k B^k(y^k, f^k(\y))=0$ for all $\y$. Specifically,
\begin{equation}
\sum_k B^k(Y^k_i, F^k_i)=0
\label{eq: Bid}
\end{equation}

Inserting \eqref{eq: scheme} into
\[I(\tilde{y}) = \sum_k B^k(\tilde{y}_k, \tilde{y}_k)
\]
we get
\[
\begin{aligned}
I(\tilde{\y}) =& \sum_k B^k\left(y^k +h\sum_i b^k_i F^k_i, y^k+h\sum_j b^k_j F^k_j\right)\\
            =&I(\y) + h \sum_{j}\sum_k b^k_j B^k\left(y^k,  F^k_j\right) + h \sum_{i}\sum_k b^k_i B^k\left(F^k_i, y^k\right)\\
            &+ h^2 \sum_{i,j}\sum_k b^k_ib^k_j B^k\left( F^k_i, F^k_j\right)\\
\end{aligned}
\]
The trick now is to substitute $y_k=Y^k_i-h\sum_ja^k_{ij} F^k_j$ to get matching terms.
\[\begin{aligned}
I(\tilde{\y}) =& I(\y) + h\sum_j\sum_k b^k_j B^k\left(Y^k_j, F^k_j\right) + h\sum_{i}\sum_{k} b^k_i B^k\left(F^k_i, Y^k_i\right)\\
            &+ h^2 \sum_{i,j}\sum_k (b^k_i b^k_j -b^k_ja^k_{ji} -
            b^k_ia^k_{ij}) B^k\left(F^k_i, F^k_j\right)\\
            =&I(\y) + 2h\sum_{i}\sum_{k} b^k_i B^k\left(Y^k_i, F^k_i\right)\\
            &+ h^2 \sum_{i,j}\sum_k (b^k_i b^k_j -b^k_ja^k_{ji} -
            b^k_ia^k_{ij}) B^k\left(F^k_i, F^k_j\right).
            \end{aligned}
\]
We see that under the assumption $b^k_i b^k_j -b^k_ja^k_{ji} - b^k_ia^k_{ij}=0$,
the $\mathcal{O}(h^2)$ term disappears. For the $\mathcal{O}(h)$ term, we see
that if $b^k_i=b_i$ is constant in $k$, then
\[
2h\sum_{i}\sum_{k} b^k_i B^k\left(Y^k_i, F^k_i\right)= 2h\sum_i b_i = \sum_k B^k\left( Y^k_i, F^k_i \right)=0,
\]
due to \eqref{eq: Bid}
\end{proof}

Now, let equation \eqref{eq: ode} be a Hamiltonian system given by $H(y, z)$ of the form
\begin{equation}
\frac{\ud y^k}{\ud t} = \mat{J}_k^{-1} \frac{\partial H}{\partial {y_k}}\\
\label{eq: Ham}
\end{equation}
where $\mat{J}_k$ are non-singular, skew-symmetric matrices.

Applying the above Theorem to the variational equation yields the following corollary
\begin{corollary}
If the coefficients satisfy that
\begin{enumerate}[(i)]
\item $b^k_ib^k_j = b^k_i a^k_{ij} + b^k_j a_{ji}$ \qquad for all $i,j,k$ and
\item $b^1_i=b^2_i= \dotsb = b^N_i$ \qquad for all $i$.
\end{enumerate}
Then the scheme is symplectic when applied to the system \eqref{eq: Ham}
\end{corollary}

We now turn to the result actually needed in this paper, where each component is integrated with a partitioned Runge--Kutta method.

Consider a Hamiltonian system
\begin{equation}\begin{aligned}
\frac{\ud q^k}{\ud t} &= \frac{\partial H}{\partial p^k} = f^k(q,p)\\
\frac{\ud p^k}{\ud t} &= -\frac{\partial H}{\partial q^k} = g^k(q,p) & k&=1,\dotsc, N
\end{aligned}
\label{eq: ode2}
\end{equation}
We integrate the system with an integrator of the form
\begin{equation}
\begin{aligned}
\tilde{q}^k&= q^k + h\sum_{i=1}^s b^k_i F^{k}_i & k&=1,\dotsc,N\\
Q_i^k &= q^k + h\sum_{j=1}^s a^k_{ij} F^{k}_j & k&=1,\dotsc,N, \quad i=1,\dotsc,s\\
F^{k}_i &= f^k(Q^k_i, P^k_i) & k&=1,\dotsc,N, \quad i=1,\dotsc,s\\
\tilde{p}^k&= p^k + h\sum_{i=1}^s \hat{b}^k_i G^{k}_i & k&=1,\dotsc,N\\
P_i^k &= p^k + h\sum_{j=1}^s \hat{a}^k_{ij} G^{k}_j & k&=1,\dotsc,N, \quad i=1,\dotsc,s\\
G^{k}_i &= g^k(Q^k_i, P^k_i) & k&=1,\dotsc,N, \quad i=1,\dotsc,s\\
\end{aligned}
\label{eq: scheme2}
\end{equation}

\begin{theorem}
\label{thm: PRK}
Assume we apply the scheme \eqref{eq: scheme} to the Hamiltonian system \eqref{eq: ode2}. If the coefficients satisfy
\begin{enumerate}[(i)]
\item $\hat{b}^k_ib^k_j = \hat{b}^k_i a^k_{ij} + b^k_j \hat a_{ji}$ \qquad for all $i,j,k$,
\item $b^{k}_i = \hat{b}^k_i$ \qquad for all $k,i$ and
\item $b^1_i=b^2_i= \dotsb = b^N_i$ \qquad for all $i$.
\end{enumerate}
then the integrator is symplectic.
\end{theorem}

The proof of the theorem is analogous to the proof of \ref{thm: RK1}, and is omitted.
%

\section{Numerical experiments}
Numerical tests were done on a simplified version of the spin-lattice-electron
equations.
In this system, the position and velocity of each particle is
confined to a one-dimensional space.
Furthermore, we use periodic boundaries in space and only consider forces
between neighbouring particles.

The total Hamiltonian is
\[
H(\w, \q,\p) = T_L(\p)+ U_L(\q)+ H_m(\w, \q)
\]
where
\[\begin{gathered}
T_L(p) = \sum_{i=1}^N \frac{p_i^2}{2m_i}, \qquad U_L(\q) =  \sum_{i=1}^n
U(q_{i+1}-q_i),\\
H_m(\w,\q)= \sum_{i=1}^n J(q_{i+1}-q_i)\mathbf{z}_i^\tran \mathbf{z}_{i+1}.
\end{gathered}
\]
To effectuate the periodic boundary, we define $\mathbf{z}_{N+1}= \mathbf{z}_1$ and $q_{N+1} = q_1+L$, where $L$ is the period.

The intermolecular potential is the Lennart--Jones' potential
\[
U(r) = U_0\left[\left(\frac{r_m}{r}\right)^{12}-2\left(\frac{r_m}{r}\right)^6\right]
\]
where $r_m$ is the ``rest distance'' i.e. the distance at which $U$ is minimal,
and $U_0$ is a positive scalar which controls the strength of the interaction.

The magnetic force strength is a cubic function, of the same type as given by Ma, Woo and Dudarev \cite{MaDuWo08}
\[
J(r) =J_0 \cdot \left(1-\frac{r}{r_c}\right)^3\cdot \Theta(r_c-r),
\]
where $\Theta$ is the Heaviside step function, and $r_c$ is a cut-off distance. $J_0$ is a scalar which controls the strength of the magnetic interaction.

In the numerical tests performed, the values were set to,
\[
\begin{aligned}
L&=N=30\\
m_i&=1,\\
U_0 &=1,&  r_m &= 1,\\
J_0 &=10, & r_c &=1.5
\end{aligned}
\]

For initial data, we set
\[
q_k = k, \qquad p_k = 0, \qquad \w_k = a_k\begin{bmatrix}0.8 \cos \left(\frac{2\pi k}{n}\right) + 0.5\sin\left(\frac{4\pi k}{n}\right)\\
0.8 \sin \left(\frac{2\pi k}{n}\right) + 0.5\cos\left(\frac{4\pi k}{n}\right) \\
   1
   \end{bmatrix}
\]
where $a_k$ is chosen so that $\norm{\w_k}=1$.

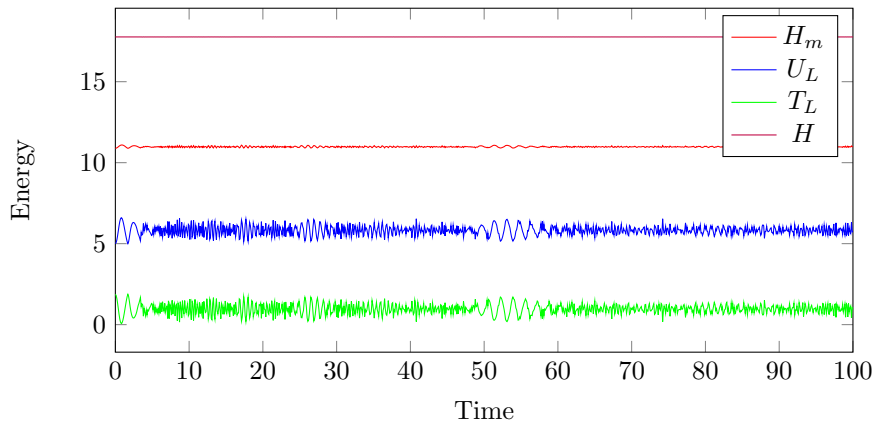
\begin{figure}
  \begin{tikzpicture}
    \begin{axis}[
      width = 0.9\textwidth,
      height = 0.3\textheight,
      xlabel =  {Time},
      ylabel = {Energy},
      xmin = 0, xmax = 100]
      \addplot[ mark=none, color=red] table [x= T, y=Hmag]{data/energies_skipped.dat};
      \addlegendentry{$H_{m}$}
      \addplot[ mark=none, color=blue] table [x= T, y expr=\thisrow{Hpot}+35]{data/energies_skipped.dat};
      \addlegendentry{$U_{L}$}
      \addplot[ mark=none, color=green] table [x= T, y=Hkin]{data/energies_skipped.dat};
      \addlegendentry{$T_L$}
      \addplot[ mark=none, color=purple] table [x= T, y expr = \thisrow{Hmag}+\thisrow{Hpot}+\thisrow{Hkin}+35]{data/energies_skipped.dat};
      \addlegendentry{$H$}
      \end{axis}
  \end{tikzpicture}
  \caption{Energy behaviour over time}
  \label{fig: energy}
\end{figure}

Figure \ref{fig: energy} shows the long term behaviour of the energy terms $T_L,
U_L$ and $H_m$ as well as their sum $H$. (A constant term has been added to
$U_L$ to improve readability). The figure shows that while the energy is
exchanged between the terms with an amplitude on the order of $\bigO(1)$, the variation
in the sum is much smaller, on the order of $\bigO(10^{-3})$.

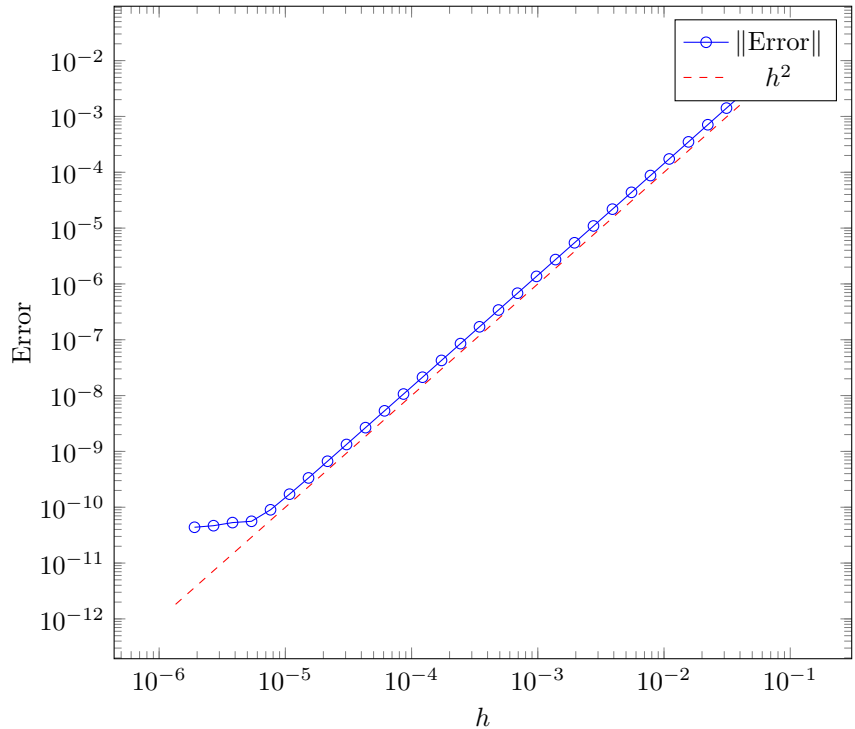
\begin{figure}
  \begin{tikzpicture}
    \begin{loglogaxis}[
      width = 0.9\textwidth,
      height = 0.5\textheight,
      xlabel = {$h$},
      ylabel = {Error}]
      \addplot[mark = o, color=blue] table [x=H, y=Err]{data/Errorplot.dat};
      \addlegendentry{$\norm{\text{Error}}$}
      \addplot[mark=none, color=red, dashed, domain=1.35e-6:0.1, samples=2] {x^2};
      \addlegendentry{$h^2$}
    \end{loglogaxis}
  \end{tikzpicture}
  \caption{Pseudoerror as function of stepsize $h$}
  \label{fig: error}
 \end{figure}
 
The integrator was also tested with various stepsizes $h$ from $h=\frac{1}{16}$
down to $h=2^{-19}$ over the time interval $[0,1]$ and the final values were
compared with the same integrator using stepsize $h=2^{-20}$.
The resulting pseudoerrors are ploted in Figure \ref{fig: error}.
The plot shows the apparent second order of the integrator.

\subsection*{Acknowledgements}
This project has received funding from the Knut and Alice Wallenberg Foundation grant agreement KAW~2014.0354.
The author would like to thank Klas Modin and Olivier Verdier for very useful discussions and comments about collective integrators and the spherical midpoint method. Klas Modin should also be thanked for supplying the Julia code that formed the basis of the implementation used for the numerical tests.

\printbibliography

\end{document}